\documentclass[11pt]{amsart}

\usepackage{epigamath}

\usepackage[english]{babel}

\numberwithin{equation}{section}

\usepackage{subcaption} 
\usepackage[shortlabels]{enumitem} 
\setlist[enumerate,1]{label={\rm(\alph*)}, ref={\rm\alph*}}

\usepackage{amsmath}
\usepackage{amsfonts}
\usepackage{amssymb}
\usepackage{mathtools}
\usepackage{physics}
\usepackage{mathrsfs} 
\usepackage{amsthm} 
\usepackage{tikz}

\usepackage{verbatim}

\usepackage{amsthm}

\theoremstyle{plain}
\newtheorem{thm}{Theorem}[section]
\newtheorem{prop}[thm]{Proposition}
\newtheorem{lem}[thm]{Lemma}
\newtheorem{cor}[thm]{Corollary}

\newtheorem*{thmA}{Theorem A}
\newtheorem*{propB}{Proposition B}

\newtheorem*{Kly}{Klyachko's classification theorem}

\theoremstyle{definition}

\theoremstyle{remark}

\newtheorem{exmp}[thm]{Example}

\newcommand{\Spec}{\operatorname{Spec}}

 \newcommand{\gC}{\Gamma}
 \newcommand{\gD}{\Delta}
\newcommand{\gi}{\iota} 
\newcommand{\gl}{\lambda} 
 \newcommand{\gO}{\Omega}
\newcommand{\gs}{\sigma} 
 
\newcommand{\bbC}{\mathbb{C}}

\newcommand{\bbP}{\mathbb{P}}

\newcommand{\bbR}{\mathbb{R}}

\newcommand{\bbZ}{\mathbb{Z}}

\newcommand{\calE}{\mathcal{E}}
\newcommand{\calF}{\mathcal{F}}

\newcommand{\calL}{\mathcal{L}}

\newcommand{\calO}{\mathcal{O}}

\newcommand{\calT}{\mathcal{T}}

\newcommand{\vol}{\operatorname{vol}}
\newcommand{\Span}{\operatorname{span}}
\newcommand{\Cone}{\operatorname{Cone}}
\newcommand{\Ext}{\operatorname{Ext}}

\newcommand{\rvline}{\hspace*{-\arraycolsep}\vline\hspace*{-\arraycolsep}}

\newcommand{\lra}{\longrightarrow}

\EpigaVolumeYear{7}{2023} \EpigaArticleNr{16} \ReceivedOn{April 25, 2022}
\InFinalFormOn{March 22, 2023}
\AcceptedOn{April 15, 2023}

\title{Affine subspace concentration conditions}
\titlemark{Affine subspace concentration conditions}

\author{Kuang-Yu Wu}
\address{Department of Mathematics, Statistics, and Computer Science, University of Illinois at Chicago, 322 Science and Engineering Offices (M/C 249), 851 S. Morgan Street, Chicago, IL 60607, USA}
\email{kwu33@uic.edu}

\authormark{K-Y.~Wu}

\AbstractInEnglish{We define a new notion of affine subspace concentration conditions for lattice polytopes and prove that these conditions hold for smooth and reflexive polytopes with barycenter at the origin.  Our proof involves considering the slope stability of the canonical extension of the tangent bundle by the trivial line bundle with  the extension class $c_1(\calT_X)$ on Fano toric varieties.}



\acknowledgement{The author was partially supported by the NSF grant DMS-1707661.}

\begin{document}



\maketitle

\begin{prelims}

\DisplayAbstractInEnglish

\bigskip

\DisplayKeyWords

\medskip

\DisplayMSCclass

\end{prelims}


\newpage

\setcounter{tocdepth}{1}

\tableofcontents


\section{Introduction}

{
Subspace concentration conditions for polytopes and more generally for
convex bodies have been of interest to convex geometers recently.  One
main reason is its close relation to the logarithmic Minkowski
problem, which asks what finite Borel measure on $S^{n-1}$ is the
cone-volume measure of a convex body in~$\bbR^n$.  In terms of
polytopes, given $v_1,\ldots,v_m\in S^{n-1}$ and
$V_1,\ldots,V_m\in(0,\infty)$, the problem asks if there is a polytope
$P\subseteq\bbR^n$ that contains the origin and has exactly $m$ facets
$F_1,\ldots,F_m$ such that the normal vector of each facet $F_k$ is
exactly~$v_k$ and the volume of the cone over each $F_k$ with vertex
at the origin is exactly~$V_k$.  As it turns out, subspace
concentration conditions give a sufficient condition for a positive
answer to the logarithmic Minkowski problem.  See, for example,
\cite{BLYZ13,Zhu14,CLZ18}.  }

{
In this paper,
we introduce new conditions of the same type,
which we call the \textit{affine subspace concentration conditions},
and we prove that these conditions hold for smooth and reflexive lattice polytopes with barycenter at the origin.

\begin{thmA}[$=\,$Theorem~\ref{thm:ASCC}, Affine subspace concentration conditions]
Let $P\subseteq\bbR^n$ be a smooth and reflexive lattice polytope with barycenter at the origin, and denote its facets by $P_1,\ldots,P_m$.  For each facet $P_k$, let $v_k\in\bbZ^n$ be its primitive inner normal vector, and let $\vol(P_k)$ be its lattice volume with respect to the intersection of\, $\bbZ^n$ and the affine span of\, $P_k$.
	
Then, for every proper affine subspace $A\subsetneq\bbR^n$, the following inequality holds: 
	$$
	\frac{1}{\dim A +1}\sum_{k\,:\,v_k\in A}\vol(P_k)
	\leq
	\frac{1}{n+1}\sum_{k=1}^m \vol(P_k).
	$$
In addition, whenever equality holds for some $A$, equality also holds for some affine subspace $A'$ complementary to~$A$.
\end{thmA}
}


{ This theorem is inspired by \cite{HNS19}, in which the subspace concentration properties are shown for polytopes with the same assumptions.

\begin{prop}[\textit{cf.}~\protect{\cite[Corollary 1.7]{HNS19}}]
Let $P\subseteq\bbR^n$ be a smooth and reflexive lattice polytopes with barycenter at the origin.  Then, for every proper linear subspace $F\subsetneq\bbR^n$, the following inequality holds: 
	$$
	\frac{1}{\dim F}\sum_{k\,:\,v_k\in F}\vol(P_k)
	\leq
	\frac{1}{n}\sum_{k=1}^m \vol(P_k).
	$$
In addition, whenever equality holds for some $F$,  equality also holds for some subspace $F'$ complementary to~$F$.
\end{prop}

In fact, our result is neither stronger nor weaker than the original subspace concentration conditions; specifically, our result gives something new whenever $A$ is \textit{not} a linear subspace of $\bbR^n$.

Similar subspace concentration conditions have also been studied in more general settings.  The subspace concentration conditions are proved in \cite{HL14} for centered polytopes and more generally in \cite{BH16} for centered convex bodies.  }

\begin{exmp}
Let $P$ be a reflexive lattice triangle.
The affine subspace concentration conditions hold for $P$ if and only if the lattice length of each side of $P$ does not exceed one third of the perimeter and the total lattice length of any two sides of $P$ does not exceed two thirds of the perimeter.  This is equivalent to all three sides of $P$ having the same lattice length.
	
Figure~\ref{fig:1a} is the unique smooth and reflexive lattice triangle, up to change of basis.  Its barycenter is clearly at the origin.  The affine subspace concentration conditions are indeed satisfied since each of its sides has lattice length $3$.
	
Figure~\ref{fig:1b} is one example of polytopes for which the affine subspace concentration conditions fail.  The lattice lengths of its sides are $1$, $1$, and $2$.  Note that the triangle is reflexive but not smooth, and its barycenter is not at the origin.
	
We do not claim that the converse to our main theorem holds.  In fact, Figure~\ref{fig:1c} is a lattice triangle that has its barycenter at the origin and is reflexive but not smooth.  However, it satisfies the affine subspace concentration conditions since all three sides has lattice length $1$.
	
	\begin{figure}[h]
		\begin{subfigure}{0.33\textwidth}
			\begin{center}
				\begin{tikzpicture}
					\filldraw [black] 	(0,0) circle (2pt)
					(-1,0) circle (1pt) 
					(1,0) circle (1pt)
					(2,0) circle (1pt)
					(-1,2) circle (1pt)
					(0,2) circle (1pt)
					(1,2) circle (1pt)
					(2,2) circle (1pt)
					(-1,1) circle (1pt)
					(0,1) circle (1pt)
					(1,1) circle (1pt)
					(2,1) circle (1pt)
					(-1,-1) circle (1pt)
					(0,-1) circle (1pt)
					(1,-1) circle (1pt)
					(2,-1) circle (1pt);
					\draw (-1,2) -- (2,-1) -- (-1,-1) -- (-1,2);
				\end{tikzpicture}
			\end{center}
			\caption{}
			\label{fig:1a}
		\end{subfigure}
		\begin{subfigure}{0.32\textwidth}
			\begin{center}
				\begin{tikzpicture}
					\filldraw [white] 
					(0,1.5) circle (1pt)
					(0,-1.5) circle (1pt);
					\filldraw [black] 
					(0,0) circle (2pt)
					(-1,0) circle (1pt)
					(1,0) circle (1pt)
					(-1,1) circle (1pt)
					(0,1) circle (1pt)
					(1,1) circle (1pt)
					(-1,-1) circle (1pt)
					(0,-1) circle (1pt)
					(1,-1) circle (1pt);
					\draw (0,1) -- (1,-1) -- (-1,-1) -- (0,1);
				\end{tikzpicture}
			\end{center}
			\caption{}
			\label{fig:1b}
		\end{subfigure}
		\begin{subfigure}{0.33\textwidth}
			\begin{center}
				\begin{tikzpicture}
					\filldraw [white] 
					(0,1.5) circle (1pt)
					(0,-1.5) circle (1pt);
					\filldraw [black]
					(0,0) circle (2pt)
					(-1,0) circle (1pt)
					(1,0) circle (1pt)
					(-1,1) circle (1pt)
					(0,1) circle (1pt)
					(1,1) circle (1pt)
					(-1,-1) circle (1pt)
					(0,-1) circle (1pt)
					(1,-1) circle (1pt);
					\draw (0,1) -- (1,0) -- (-1,-1) -- (0,1);
				\end{tikzpicture}
			\end{center}
			\caption{}
			\label{fig:1c}
		\end{subfigure}
		\caption{Reflexive lattice triangles}
		\label{fig:a}
	\end{figure}
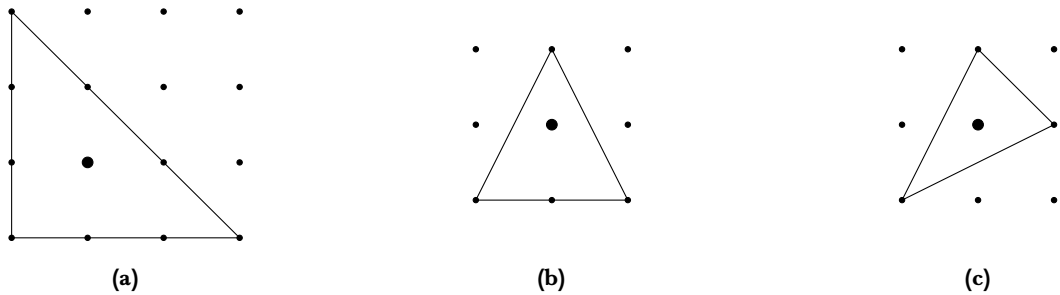
\end{exmp}

\subsection{Idea of the proof}

{ As in \cite{HNS19}, our proof of the main result relies on toric geometry as well as some K\"ahler geometry.

First, let $X$ be the smooth Fano toric variety corresponding to the polytope $P$.  We consider the canonical extension
$$
0\lra\calO\lra\calE\lra\calT_X\lra 0
$$
with the extension class $c_1(\calT_X)\in\Ext^1(\calT_X,\calO)$.  The main work we will do is in Section~\ref{sec:calE}, where we show that this rank $n+1$ extension $\calE$ can be made a toric vector bundle and compute its corresponding filtrations in the sense of Klyachko's classification theorem of toric vector bundles.

\begin{propB}[$=\,$Proposition~\ref{prop:fltrE}]
There is a $T$-action on the extension $\calE$ that makes it a toric vector bundle whose corresponding filtrations $E^{\rho}(i)$ of the $(n+1)$-dimensional\, $\bbC$-vector space $E\cong\bbC^{n+1}$ are given by
	$$
	E^{\rho_k}(i)
	=
	\begin{cases}
		E,							&	i\leq 0,	\\
		\Span_{\bbC}\{(v_k,-1)\},	&	i=1,		\\
		0,							&	i>2.
	\end{cases}
	$$
\end{propB}

Then we insert some K\"ahler geometry.  Since $P$ has its barycenter at the origin, it is known that $X$ admits a K\"ahler--Einstein metric; see \cite{WZ04,Mab87}.  A theorem of Tian \cite{Tia92} then implies that the canonical extension $\calE$ admits a Hermitian--Einstein metric, and by the easy direction of the Donaldson--Uhlenbeck--Yau theorem, this implies the slope polystability of $\calE$ with respect to $\calO_X(-K_X)$.

Combining this with Proposition B and a proposition in \cite{HNS19} which characterizes the stability of toric vector bundles, we prove the affine subspace concentration conditions in Section~\ref{sec:ASCC}.  }

\subsection*{Acknowledgments}
I would also like to thank my advisor Julius Ross for setting this project and for discussions about it.

\section{Preliminaries}
{ Here we list some definitions and facts regarding toric varieties and toric vector bundles that we will use in this article.  One may refer to \cite{CLS11,Ful93} for more details about toric varieties and to \cite{Kly90,Per04,Pay08} for more details about toric vector bundles.  }
\subsection{Toric varieties}

We work throughout over $\bbC$.  By a toric variety, we mean an irreducible and normal algebraic variety $X$ containing a torus $T\cong(\bbC^*)^n$ as a Zariski open subset such that the action of $T$ on itself (by multiplication) extends to an algebraic action of $T$ on $X$.

Let $M$ be the group of characters of $T$ and $N$ the group of $1$-parameter subgroups of $T$.  Both $M$ and $N$ are lattices of rank $n$ (equal to the dimension of $T$), \textit{i.e.}, isomorphic to $\bbZ^n$, and they are dual to each other.  Explicitly, the character corresponding to $u=(a_1,\ldots,a_n)\in\bbZ^n$ is given by
$$
\chi^u\colon T\lra\bbC^*,\quad
(t_1,\ldots,t_n)\longmapsto t_1^{a_1}\cdots t_n^{a_n},
$$
and the $1$-parameter subgroup corresponding to $v=(b_1,\ldots,b_n)\in\bbZ^n$ is given by
$$
\gl^v\colon\bbC^*\lra T,\quad
t\longmapsto (t^{b_1},\ldots,t^{b_n}).
$$
The pairing of $M$ and $N$ is denoted by $\langle\cdot,\cdot\rangle\colon M\times N\to\bbZ$.  We have $\chi^u(\gl^v(t))=t^{\langle u,v\rangle}$ for all $u\in M$ and $v\in N$.

Every toric variety $X$ is associated to a fan $\gD$ in $N_{\mathbb{R}}:=N\otimes_{\mathbb{Z}}\mathbb{R}\,(\cong\bbR^n)$.  Each cone $\gs$ corresponds to a $T$-invariant open affine subset of $X$, denoted by $U_{\gs}$.  One has $U_{\gs} \cong \Spec \bbC[\gs^{\vee} \cap M]$, where $\gs^{\vee}$ is the dual cone of $\gs$ in $M_{\mathbb{R}} := M\otimes_{\mathbb{Z}}\mathbb{R}$

A fan $\gD$ is said to be \textit{complete} if it supports on the whole $N_{\bbR}$ and is said to be \textit{smooth} if every cone in $\gD$ is generated by a subset of a $\bbZ$-basis of $N$.  A toric variety $X$ is complete if and only if its associated fan $\gD$ is complete, and $X$ is smooth if and only if $\gD$ is smooth.

There is an inclusion-reversing bijection between the cones $\sigma\in\Delta$ and the $T$-orbits in $X$.  Let $O_{\sigma}$ be the orbit corresponding to $\sigma$.  The dimension of $\sigma$ is equal to the codimension of $O_{\sigma}$ in $X$.

Fix a cone $\gs\in\gD$, and let $v$ be a lattice point in the relative interior of $\gs$.  The $1$-parameter subgroup $\gl^v$ approaches a point $p_{\gs}$ in the orbit $O_{\gs}$ as $t\to 0$.  This point $p_{\gs}$ is called the \textit{distinguished point} corresponding to $\gs$.

Given a $1$-dimensional cone $\rho\in\gD$, the closure of $O_{\rho}\subseteq X$ is a Weil divisor, which we denote by $D_{\rho}$.  Suppose $X$ is smooth.  The canonical divisor $K_X$ of $X$ is equal to $-\sum_{\rho}D_{\rho}$, where the sum is taken over all $1$-dimensional cones $\rho\in\gD$.

\subsection{Polytopes and toric varieties}

Let $M_{\bbR}:=M\otimes_{\bbZ}\bbR\cong\bbR^n$.  A lattice polytope $P$ in $M_{\bbR}$ is the convex hull in $M_{\bbR}$ of finitely many points in~$M$.  The dimension of $P$ is the dimension of the affine span of $P$.  When $\dim P=\dim M_{\bbR}$, we say that $P$ is full-dimensional.

Let $P\subseteq M_{\bbR}$ be a full-dimensional lattice polytope, and let $P_1,\ldots,P_m$ be the \textit{facets} of $P$, \textit{i.e.}, codimension $1$ faces of $P$.  For each facet $P_k$, there exist a unique primitive lattice point $v_k\in N$ and a unique integer $c_k\in\bbZ$ such that
$$
P_k = \{u\in P\,|\,\langle u,v_k\rangle =-c_k\}
$$
and $\langle u,v_k\rangle\geq -c_k$ for all $u\in P$.

Denote the (inner) normal fan of $P$ by $\gD_P$.  The $1$-dimensional cones in $\gD_P$ are exactly the rays generated by $v_k$.  The toric variety $X_{\gD_P}$ associated to $\gD_P$ is called the toric variety of $P$ and denoted by $X_P$.  Let $D_k$ be the divisor corresponding to the $1$-dimensional cone generated by $v_k$.  Then we may define a divisor on $X$ by $D_P:=\sum_{k=1}^mc_kD_k$.  Such a divisor $D_P$ is necessarily ample.

This process is reversible, and there is a 1-to-1 correspondence between full-dimensional lattice polytopes $P\subseteq M_{\bbR}$ and  pairs $(X,D)$ of a complete toric variety $X$ together with an ample $T$-invariant divisor $D$ on $X$.

Fix a vertex $u$ of $P$.  Let $u_1,\ldots,u_n$ be the next lattice points on the $n$ edges (\textit{i.e.}, $1$-dimensional faces of $P$) containing $u$.  We say that $P$ is \textit{smooth} if $\{u_1-u,\ldots,u_n-u\}$ is a $\bbZ$-basis of $M$ for all vertices $u$ of $P$.  The toric variety $X_P$ of $P$ is smooth if and only if $P$ is smooth.

A polytope $P$ is called \textit{reflexive} if the integers $c_k$ ($k=1,\ldots,m$) defined above are all equal to $1$, and a smooth variety $X$ is said to be \textit{Fano} if its anticanonical divisor $-K_X$ is ample.  A reflexive and smooth polytope $P$ corresponds to a smooth Fano toric variety $X=X_P$ together with its anticanonical divisor $-K_X$.

\subsection{Toric vector bundles}

A vector bundle $\pi\colon \mathcal{E}\to X$ over a toric variety $X=X_{\Delta}$ is said to be toric (or equivariant) if there is a $T$-action on $\mathcal{E}$ such that $t\circ\pi=\pi\circ t$ for all $t\in T$.

Let $\mathcal{E}\to X$ be a toric vector bundle.  Fix a cone $\gs\in\gD$, and let $U_{\gs}\subseteq X$ be the corresponding open affine subset.  The $T$-action on $\mathcal{E}$ induces a $T$-action on sections given by
$$
(t\cdot s)(p):=t\cdot (s(t^{-1}\cdot p)),
$$
where $t\in T$, $p\in U_{\gs}$, and $s\in\gC(U_{\gs},\mathcal{E})$.  A section $s\in\gC(U_{\gs},\mathcal{E})$ is called a $T$-\textit{eigensection} if there exists a $u\in M$ such that $t\cdot s=\chi^u(t)\cdot s$ for all $t\in T$.  For each $u\in M$, let $\gC(U_{\gs},\mathcal{E})_u$ be the set of $T$-eigensections $s$ with $t\cdot s=\chi^u(t)\cdot s$.

Let $t_0\in T$ be the identity and $E=\mathcal{E}_{t_0}$ the fiber over $t_0$.  The evaluation at $t_0$ defines an injection $\gC(U_{\gs},\mathcal{E})_u\hookrightarrow E$ for each $u\in M$.  Define $E^{\gs}_u\subseteq E$ to be the image of $\gC(U_{\gs},\mathcal{E})_u$ under this injection.

Given a $u'\in M$ that lies in the dual of $\gs$, defined by $\gs^{\vee}:=\{u\in M\,|\,\langle u,v\rangle\geq 0\text{ for all }v\in\gs\}$, the character $\chi^{u'}$ of $T$ extends to a regular function on $U_{\gs}$.  Then, the multiplication by $\chi^{u'}$ gives a map $\gC(U_{\gs},\mathcal{E})_u\to \gC(U_{\gs},\mathcal{E})_{u-u'}$ and in turn induces an inclusion $E^{\gs}_u\subseteq E^{\gs}_{u-u'}$.
If $u'$ further satisfies $\langle u,v\rangle=0$ for all $v\in \gs$ (\textit{i.e.}, $u'\in\gs^{\perp}$), then $-u'$ also lies in $\gs^{\vee}$, and we obtain $E^{\gs}_u=E^{\gs}_{u-u'}$.  Therefore, $E^{\gs}_u$ depends only on the class of $u$ in $M_{\gs}:=M/(M\cap\gs^{\perp})$.

Let $\rho$ be a $1$-dimensional cone in $\gD$.  In this case, $M_{\rho}$ is isomorphic to $\bbZ$.  Let $v_{\rho}\in N$ be the primitive generator of $\rho$.  For each $i\in\bbZ$, take one $u_i\in M$ such that $\langle u_i,v_{\rho}\rangle=i$, and define $E^{\rho}(i):=E^{\rho}_{u_i}$.  Then, we get a decreasing filtration of $E$
$$
\cdots\supseteq E^{\rho}(i)\supseteq E^{\rho}(i-1)\supseteq\cdots.
$$

These filtrations corresponding to the $1$-dimensional cones in fact contain all information about a toric vector bundle according to the following classification theorem of Klyachko.

\begin{Kly}[\textit{cf.}~\protect{\cite[Theorem 2.2.1]{Kly90}}] 
The category of toric vector bundles on a toric variety $X=X(\gD)$ is naturally equivalent to the category of finite-dimensional vector spaces $E$ with collections of decreasing filtrations $\{E^{\rho}(i)\}$ indexed by the $1$-dimensional cones in $\gD$, satisfying the following compatibility condition:  for each cone $\gs\in\gD$, there is a decomposition $E=\bigoplus_{[u]\in M_{\gs}}E_{[u]}$ such that
	$$
	E^{\rho}(i)=\sum_{[u]:\langle u,v_p\rangle\geq i}E_{[u]}
	$$
	for all $\rho\preceq\gs$ and $i\in\bbZ$.
\end{Kly}

Another way to consider a toric vector bundle $\calE$ is to restrict it to open affine sets.  To discuss this, let~$\gs$ be a cone in $\gD$ and $U_{\gs}\subseteq X$ the corresponding open affine set.  On $U_{\gs}$, the toric vector bundle $\calE$ splits equivariantly into a direct sum of toric line bundles whose underlying line bundles are trivial; see \cite[Proposition 2.2]{Pay08}.  Collecting non-vanishing sections of these toric line bundles, one can get a local frame consisting of $T$-eigensections.

\begin{lem} \label{lem:2.1}
Let $\rho$ be a $1$-dimensional cone contained in $\gs$ and $v_{\rho}\in \rho\cap N$ the primitive generator of\, $\rho$.  Suppose that $\calE$ has a local frame $\{s_1,\ldots,s_r\}$ on $U_{\gs}$ such that $t\cdot s_k=\chi^{u_k}(t)\cdot s_k$ for some $u_1,\ldots,u_r\in M$, and define $L_k:=\Span\{s_k(t_0)\}\subseteq E$ for each $k=1,\ldots,r$.  Then, we have
	$$
	E^{\rho}(i)
	=
	\sum_{k:\langle u_k,v_{\rho}\rangle\geq i}
	L_k.
	$$
\end{lem}

\begin{proof}
By assumption, we have
	$$
	\calE|_{U_{\gs}}
	\cong
	\bigoplus_{k=1}^r\calL_k,
	$$
where $\calL_k$ is the trivial line bundle $U_{\gs}\times\bbC$ spanned by $s_k$ equipped with the $T$-action defined by
	$$
	t\cdot (p,s_k(p))=(t\cdot p,\chi^u(t)\cdot s_k(p)).
	$$
By Klyachko's classification theorem, each toric line bundle $\calL_k$ on $U_{\gs}$ also corresponds to a collection of filtrations of $L_k$ indexed by the $1$-dimensional cones contained in $\gs$.  The filtration corresponding to $\rho$ is given by
	$$
	(L_k)^{\rho}(i)
	=
	\begin{cases}
		L_k,	&	i\leq \langle u_k,v_{\rho}\rangle,	\\
		0,	&	i>\langle u_k,v_{\rho}\rangle.
	\end{cases}
	$$
Then, adding these filtrations together, we obtain the desired filtration of $E$.
\end{proof}

\section{The canonical extension} \label{sec:calE}
{ { Let $X$ be a smooth Fano variety of dimension $n$.  The first Chern class $c_1(X)$ of $X$ lies in the $(1,1)$-Dolbeault cohomology $H^{1,1}(X)$, which is naturally isomorphic to the sheaf cohomology $H^1(X,\gO^1_X)$ by Dolbeault's theorem.  The cotangent bundle $\gO^1_X$ is the dual of the tangent bundle $\calT_X$, so we have
$$
H^1(X,\gO^1_X)=H^1(X,\calT_X^*)
\cong
\Ext^1(\calO_X,\calT_X^*)
\cong
\Ext^1(\calT_X,\calO_X)
$$
by for example \cite[Propositions 6.3 and 6.7]{Har77}.  Thus, $c_1(X)$ can be viewed as an extension class in $\Ext^1(\calT_X,\calO_X)$ and in turn corresponds to an extension $\calE$ of the tangent bundle $\calT_X$ by the structure sheaf $\calO_X$ given by
$$
0\lra\calO_X\lra\calE\lra\calT_X\lra 0.
$$

In this section, we will show that if $X$ is toric, then $\calE$ can be made a toric vector bundle, and we will compute its corresponding filtrations in the sense of Klyachko's classification theorem of toric vector bundles.  }

{ Let $\gD_X$ be the corresponding fan of $X$ in $N_{\bbR}\cong\bbR^n$.  Denote the $1$-dimensional cones in $\gD_X$ by $\rho_1,\ldots,\rho_m$, and let $v_1,\ldots,v_m\in N$ be the primitive generators of these $1$-dimensional cones, respectively.  }

\begin{prop} \label{prop:fltrE}
There is a $T$-action on the extension $\calE$ that makes it a toric vector bundle whose corresponding filtrations $E^{\rho}(i)$ of the $(n+1)$-dimensional\, $\bbC$-vector space $E:=N_{\bbC}\oplus\bbC\cong\bbC^{n+1}$ are given by
	$$
	E^{\rho_k}(i)
	=
	\begin{cases}
		E,							&	i\leq 0,	\\
		\Span_{\bbC}\{(v_k,-1)\},	&	i=1,		\\
		0,							&	i>2.
	\end{cases}
	$$
\end{prop}

{
To prove this, we will use the following construction of the extension $\calE$.  First, embed $X$ into a projective space $\bbP^N$ with the very ample divisor $c_1(\calT_X)=-K_X$.  Next, further embed $\bbP^N$ into $\bbP^{N+1}$.  Fix a point $\tilde{p}_0\in\bbP^{N+1}\backslash\bbP^N$, and let $Y\subseteq\bbP^{N+1}$ be the cone over $X$ with vertex $\tilde{p}_0$.  Then $\calE$ will be exactly  the restriction of the logarithmic tangent sheaf $\calT_Y(-\log X)$ of $Y$ to $X$.  A proof of this can be found in \cite[Proposition~3.3]{Wah76}.

In the following, we will first construct $Y$ as a toric variety and get an induced toric action on the logarithmic tangent sheaf.  Then, we will make $\calE$ a toric vector bundle.  } }

\subsection{The logarithmic tangent sheaf}

{ Since $X$ is a smooth and Fano toric variety, its anticanonical divisor $-K_X$ is very ample.  Let $P$ be the polytope corresponding to $(X,-K_X)$.  Then the vector space $\gC(X,\calO_X(-K_X))$ of the global sections of $\calO_X(-K_X)$ is spanned by the characters corresponding to the lattice points in $P$, \textit{i.e.}, $\gC(X,-K_X)=\bigoplus_{u\in P\cap M}\bbC\cdot\chi^u$.

Let $N=\dim\gC(X,-K_X)-1$.  Then the embedding $X\hookrightarrow\bbP^N$ can be explicitly given by $x\mapsto [\chi^u(x)]_{u\in P\cap M}$.  In particular, the restriction of the embedding to the torus $T\subseteq X$ is given by $(t_1,\ldots,t_n)\mapsto [\chi^u(t_1,\ldots,t_n)]_{u\in P\cap M}$.

Without loss of generality, we may assume that the embedding of $\bbP^N$ in $\bbP^{N+1}$ is given by
$$
[Z_0:\cdots:Z_N]\longmapsto[Z_0:\cdots:Z_N:0]
$$
so that $\bbP^N$ is identified with the hyperplane in $\bbP^{N+1}$ where the last coordinate vanishes.  Also, we may set $\tilde{p}_0=[0:\cdots:0:1]$.  Then the cone over $X\subseteq\bbP^N$ with vertex $\tilde{p}_0$ is given by
$$
Y=\left\{[x:Z_{N+1}]\in\bbP^{N+1}\,\big|\,
x\in X,Z_{N+1}\in\bbC\right\}\cup\{\tilde{p}_0\}.
$$
There is a dense $(n+1)$-dimensional torus $\tilde{T}:=T\times\bbC^*$ in $Y$, and we may let it act on $Y$ by
$$
(t_1,\ldots,t_n,t_{n+1})\cdot[x:Z_{N+1}]:=
[(t_1,\ldots,t_n)\cdot x:t_{n+1}Z_{N+1}].
$$
We have now made $Y$ a toric variety.  }

{ We will denote the character lattice of $\tilde{T}$ by $\tilde{M}\cong M\oplus\bbZ$.  We also let $\tilde{N}$ be the dual lattice of $\tilde{M}$ and define $\tilde{N}_{\bbR}:=\tilde{N}\otimes_{\bbZ}\bbR$.  }

\begin{lem} \label{lem:fan of Y}
The fan $\gD_Y$ corresponding to $Y$ is one in $\tilde{N}_{\bbR} \cong N_{\bbR}\oplus\bbR \cong \bbR^{n+1}$ that contains exactly the following cones: 
	\begin{enumerate}
		\item the origin, and the $1$-dimensional cone generated by $(0 ,\ldots , 0 , 1)$;
		\item for each subset $\{v_{k_1} ,\ldots, v_{k_d}\} \subseteq \{v_1 ,\ldots , v_m\}$ that generates a cone $\gs$ in the fan $\gD_X$ of $X$, a $d$-dimensional cone $\gs_{-}$ generated by $\{ (v_{k_l} , -1) \, | \, l = 1 ,\ldots, d \}$, and a $(d + 1)$-dimensional cone $\gs_{+}$ generated by $\{ (v_{k_l} , -1) \, | \,$ $l = 1 ,\ldots, d \} \cup \{(0 ,\ldots, 0 , 1)\}$;
		\item the $(n + 1)$-dimensional cone generated by $\{ (v_k , -1) \, | \, k = 1 ,\ldots, m \}$.
	\end{enumerate}
\end{lem}

\begin{proof}
First, to determine the 1-dimensional cones in $\gD_Y$, consider the $\tilde{T}$-orbits of codimension $1$ in $Y$, \textit{i.e.}, the $n$-dimensional orbits.  There are exactly $m+1$ of them, the torus $T$ in $X$ and $O_k\times\bbC^*$ ($k=1,\ldots,m$), where the $O_k$ are the $(n-1)$-dimensional $T$-orbits in $X$ corresponding to the $1$-dimensional cones $\rho_k$ in $\gD_X$.  For simplicity, define $\tilde{O}_k:=O_k\times\bbC^*$.
	
The distinguished points on $T$ and $\tilde{O}_k$ are, respectively, $[1:\cdots:1:0]$ and $[q_k:1]$, where $q_k\in X\subseteq\bbP^N$ is the distinguished point on $O_k$.
	
Now, consider the limits of the $1$-parameter subgroups of $\tilde{T}$ corresponding to the lattice points $(0,\ldots,0,1)$ and $(v_k,-1)$ $(k=1,\ldots,m)$ in $\tilde{N}$.  Let $f$ be the embedding $X\hookrightarrow\bbP^N$ and $g$ the composition of $f$ and the embedding $\bbP^N\hookrightarrow\bbP^{N+1}$.
	
First, we have
	$$
	\lim_{t\to 0} g\left(\gl_{\tilde{T}}^{(0,\ldots,0,1)}(t)\right)
	=\lim_{t\to 0} [1:\cdots:1:t]
	=[1:\cdots:1:0],
	$$
so that the $1$-dimensional cone corresponding to the orbit $T\subseteq Y$ is generated by $(0,\ldots,0,1)$.
	
Next, since $v_k$ generates the $1$-dimensional cone $\rho_k$ in $\gD_X$, we have
	$$
	q_k
	=\lim_{t\to 0} f\left(\gl_T^{v_k}(t)\right)
	=\lim_{t\to 0}\,\left[\chi^u\left(t^{v_k}\right)\right]_{u\in P\cap M}
	=\lim_{t\to 0}\,\left[t^{\langle u,v_k\rangle}\right]_{u\in P\cap M}.
	$$
Because $P$ is reflexive, we also have that
	$$
	\min_{u\in P\cap M}\{\langle u,v_k\rangle\}=-1.
	$$
Then, we obtain
	$$
	\lim_{t\to 0} g\left(\gl_{\tilde{T}}^{(v_k,-1)}(t)\right)
	=\lim_{t\to 0}\,\left[\left[\chi^u(t^{v_k})\right]_{u\in P\cap M}:t^{-1}\right]
	=\lim_{t\to 0}\,\left[\left[t^{\langle u,v_k\rangle}\right]_{u\in P\cap M}:t^{-1}\right]
	=[q_k:1],
	$$
which implies that $(v_k,-1)$ generates the $1$-dimensional cone corresponding to $\tilde{O}_k$.
	
More generally, fix $v_{k_1} ,\ldots, v_{k_d}$ ($d > 0$) that generate a cone $\gs \in \gD_X$.  Denote by $O_{\gs}$ the $T$-orbit in $X$ corresponding to $\gs$ and, by abuse of notation, its image in $Y$ under $X \hookrightarrow Y$.  Also, define $\tilde{O}_k:=O_k\times\bbC^* \subseteq Y$.  Clearly, $O_k , \tilde{O}_k$ are both $\tilde{T}$-orbits in $Y$.  The claim is that $O_k$ corresponds to the cone $\gs_{+}$ generated by $\{ (v_{k_l} , -1) \, | \, l = 1 ,\ldots, d \} \cup \{(0 ,\ldots, 0 , 1)\}$ and $\tilde{O}_k$ corresponds to the cone $\gs_{-}$ generated by $\{ (v_{k_l} , -1) \, | \, l = 1 ,\ldots, d \}$.
	
Define $v_{\gs} := v_{k_1} + \cdots + v_{k_d} \in N$.  Let $q_{\gs} \in X$ be the distinguished point on $O_k \subseteq X$, so that
	$$
q_{\sigma} = \lim_{t\to 0} f\left(\gl_T^{v_{\gs}}(t)\right) = \lim_{t\to 0}\,
\left[t^{\langle u,v_{\gs}\rangle}\right]_{u\in P\cap M}.
	$$
Then the distinguished points on $O_k , \tilde{O}_k \subseteq Y$ are, respectively, $[q_{\gs} : 0] , [q_{\gs} : 1]$.  Take lattice points $(v_{\gs} , 0) , (v_{\gs} , -d)$ lying in the relative interiors of $\gs_{+} , \gs_{-}$, respectively.  Then, we have
	$$
	\begin{dcases}
		\lim_{t\to 0} g(\gl_{\tilde{T}}^{(v_{\gs},0)}(t))
		=\lim_{t\to 0}\,\left[\left[t^{\langle u,v_{\gs}\rangle}\right]_{u\in P\cap M}:1\right]
		=[q_k:0],
		\\
		\lim_{t\to 0} g(\gl_{\tilde{T}}^{(v_{\gs},-d)}(t))
		=\lim_{t\to 0}\,\left[\left[t^{\langle u,v_{\gs}\rangle}\right]_{u\in P\cap M}:t^{-d}\right]
		=[q_k:1],
	\end{dcases}
	$$
where the last equality of each line is	because
	$$
	\min_{u\in P\cap M}\{\langle u,v_{\gs}\rangle\}=-d ,
	$$
where the minimum is attained on the face of $P$ corresponding to $\gs$.
	
Finally, the $(n + 1)$-dimensional cone generated by $\{ (v_k , -1) \, | \, k = 1 ,\ldots, m \}$ corresponds to the cone point $\tilde{p}_0 = [0 : \cdots : 0 : 1]$ as
	\begin{equation*}\pushQED{\qed} 
	\lim_{t\to 0} g\left(\gl_{\tilde{T}}^{(0,\ldots,0,-1)}(t)\right)
	=\lim_{t\to 0} [1:\cdots:1:t^{-1}]
	=[0 : \cdots : 0 : 1].\qedhere\popQED
	\end{equation*}
\renewcommand{\qed}{}   
\end{proof}

{
}

{ It is convenient to remove the vertex $\tilde{p}_0$ and consider $Y':=Y\backslash\{\tilde{p}_0\}$ instead of the cone $Y$.  Since $\tilde{p}_0$ is a fixed point under the $\tilde{T}$-action, $Y'$ is still a toric variety.  In addition, we know that the fan of $Y'$ is equal to $\gD_Y$ with the maximal cone
$$
\tilde{\gs}:=\Cone\left((v_1,-1),\ldots,(v_m,-1)\right)
$$
corresponding to the fixed point $\tilde{p}_0$ removed; \textit{i.e.}, $\gD_{Y'}=\gD_Y\backslash\{\tilde{\gs}\}$.  One can easily see that $Y'$ is smooth either by checking that the fan $\gD_{Y'}$ is smooth or by observing that $Y'$ is the total space of the line bundle $\calO_X(-K_X)$, which follows from Lemma~\ref{lem:fan of Y} together with \cite[Proposition 7.3.1]{CLS11}.

Then, instead of the logarithmic tangent sheaf $\calT_Y(-\log X)$ on $Y$, we consider its restriction on $Y'$, which is the same as $\calT_{Y'}(-\log X)$ since $\tilde{p}_0\notin X$.  Note that $\calT_{Y'}(-\log X)$ is then locally free since $Y'$ is smooth; we will view it as a vector bundle on $Y'$.  }

{ The toric action of $Y'$ induces a natural toric action on the vector bundle, which we describe as follows.

Fix an open affine set $\tilde{U}_j\subseteq Y'$ corresponding to a maximal cone $\tilde{\gs}_j$ in $\gD_{Y'}$.  The cone $\tilde{\gs}_j$ and its dual cone $(\gs_j)^{\vee}$ must be of the form
$$
\begin{cases}
	\tilde{\gs}_j
	=\Cone\left((v_{k_1},-1),\ldots,(v_{k_n},-1),(0,\ldots,0,1)\right),
	\\
	(\tilde{\gs}_j)^{\vee}
	=\Cone\left((u_{k_1},0),\ldots,(u_{k_n},0),(u_{k_1}+\cdots+u_{k_n},1)\right),
\end{cases}
$$
where $v_{k_1},\ldots,v_{k_n}\in N$ are generators of a maximal cone $\gs_j$ in $\gD_X$ and $u_{k_1},\ldots,u_{k_n}\in M$ are generators of its dual cone $(\gs_j)^{\vee}$ with
$$
\langle u_k,v_{k'} \rangle =
\begin{cases}
	1,	&	\text{if }	k=k',		\\
	0,	&	\text{if }	k\neq k'
\end{cases}
$$
for $k,k'\in\{k_1,\ldots,k_n\}$.  Note that $\tilde{U}_j$ is isomorphic to $\bbC^{n+1}$ since $\tilde{\gs}_j$ is a smooth cone.  Let $z_1,\ldots,z_n,z$ be the local coordinates on $\tilde{U}_j$ corresponding to the generators of $(\tilde{\gs}_j)^{\vee}$ given by $(u_{k_1},0),\ldots,(u_{k_n},0),(u,1)$, respectively.  Then, explicitly, the $\tilde{T}$-action on $\tilde{U}_j\cong\bbC^{n+1}$ is given by
$$
(t_1,\ldots,t_n,t_{n+1})\cdot(z_1,\ldots,z_n,z)
=
\left(\chi^{u_{k_1}}\!(t_1,\ldots,t_n)\cdot z_1
,\ldots,
\chi^{u_{k_n}}\!(t_1,\ldots,t_n)\cdot z_n,
\chi^u(t_1,\ldots,t_n)\cdot t_{n+1}\cdot z\right).
$$

Now, the bundle $\calT_{Y'}(-\log X)$ trivializes on $\tilde{U}_j$ and has a local frame given by
$$
\left\{\frac{\partial}{\partial z_1},\ldots,\frac{\partial}{\partial z_n},z\frac{\partial}{\partial z}\right\}.
$$
In terms of the local frame, the $\tilde{T}$-action on $\calT_{Y'}(-\log X)$ is given by
$$
\begin{dcases}
(t_1,\ldots,t_n,t_{n+1})\cdot
\left.{\frac{\partial}{\partial z_l}}\right|_{\tilde{p}}
=
\chi^{u_{k_l}}\!(t_1,\ldots,t_n)\cdot
\pdv{z_l}\bigg|_{(t_1,\ldots,t_n,t_{n+1}).\tilde{p}}
&
\text{for }l=1,\ldots,n,
\\
(t_1,\ldots,t_n,t_{n+1})\cdot
\left.\left(
z{\frac{\partial}{\partial z}}
\right)\right|_{\tilde{p}}
=
\left.\left(
z{\frac{\partial}{\partial z}}
\right)\right|_{(t_1,\ldots,t_n,t_{n+1}).\tilde{p}},
\end{dcases}
$$
where $(t_1,\ldots,t_n,t_{n+1})\in\tilde{T}$ and $\tilde{p}\in\tilde{U}_j$.  The first line is simply the $\tilde{T}$-action on tangent vectors, and the second line can be obtained by first considering points $\tilde{p}$ with nonzero $z$-coordinate and then extending the action to points with zero $z$-coordinate by continuity.  }

\subsection{The canonical extension as a restriction of the logarithmic tangent sheaf}

Now, we restrict the toric vector bundle $\calT_Y(-\log X)$ back to $X$ to get the extension $\calE$.

\begin{proof}[Proof of Proposition~\ref{prop:fltrE}]
First, we define a $T$-action on $\calE$.  Define $\gi\colon T\to\tilde{T}$ by
	$$
	(t_1,\ldots,t_n)\longmapsto(t_1,\ldots,t_n,1),
	$$
and set
	$$
	t\cdot e:=\gi(t)\cdot e
	$$
for $t\in T$ and $e\in\calE$, where the left-hand side is the $\tilde{T}$-action on $\calT_Y(-\log X)$.  Note that $i\circ t=\gi(t)\circ i$ for all $t\in T$, where $i$ is the inclusion $X\hookrightarrow Y$.  Thus, this $T$-action on $\calE$ is equivariant.
	
Fix a maximal cone $\gs_j$ in $\gD_X$, and let $\tilde{\gs}_j,U_j,\tilde{U}_j,v_{k_l},u_{k_l},z_l,z$ be defined as in the previous section.  Note that $U_j\subseteq\tilde{U}_j$.  Restricting the local frame of $\calT_Y(-\log X)$ on $\tilde{U}_j$ to $U_j$ yields a local frame of $\calE$ on $U_j$ given by $\left\{\pdv{z_1},\ldots,\pdv{z_n},z\pdv{z}\right\}$.  Explicitly, the $T$-action on $\calE|_{U_j}$ is given by
	$$
	\begin{dcases}
		(t_1,\ldots,t_n)\cdot
		\pdv{z_l}\bigg|_{p}
		=
		\chi^{u_{k_l}}\!(t_1,\ldots,t_n)\cdot
		\pdv{z_l}\bigg|_{(t_1,\ldots,t_n).p}
		&
		\text{for }l=1,\ldots,n,
		\\
		(t_1,\ldots,t_n)\cdot
		\left(z\pdv{z}\right)\bigg|_{p}
		=
		\left(z\pdv{z}\right)\bigg|_{(t_1,\ldots,t_n).p}.
	\end{dcases}
	$$
Then, the sections $\pdv{z_1},\ldots,\pdv{z_n},z\pdv{z}$ are $T$-eigensections with
	$$
	\begin{dcases}
		(t_1,\ldots,t_n)\cdot
		\pdv{z_l}
		=
		\chi^{u_{k_l}}\!(t_1,\ldots,t_n)
		\pdv{z_l}
		&
		\text{for }l=1,\ldots,n,
		\\
		(t_1,\ldots,t_n)\cdot
		\left(z\pdv{z}\right)
		=
		z\pdv{z}
		=
		\chi^{(0,\ldots,0)}(t_1,\ldots,t_n)\cdot
		\left(z\pdv{z}\right).
	\end{dcases}
	$$
	
Now, we apply Lemma~\ref{lem:2.1} to get the filtrations.  Let $t_0=(1,\ldots,1)\in T$ be the identity.  Note that $ \pdv{z_1}\big|_{t_0},\ldots,\pdv{z_n}\big|_{t_0},\left(z\pdv{z}\right)\big|_{t_0} $ form a basis of $E=\calE_{t_0}$.  We may choose a set of coordinates of $E$ so that
	$$
	\pdv{z_l}\bigg|_{t_0}\!\!=(v_{k_l},-1)
	\quad\text{ for }
	l=1,\ldots,n
	\quad\text{ and }\quad
	\left(z\pdv{z}\right)\bigg|_{t_0}\!\!=(0,\ldots,0,1).
	$$
Then, we obtain
	$$
	E^{\rho_{k_l}}(i)=
	\begin{cases}
		E,								&	i\leq 0,	\\
		\Span_{\bbC}\{(v_{k_l},-1)\},	&	i=1,		\\
		0,					&	i>2
	\end{cases}
	$$
for $1$-dimensional cones $\rho_{k_1},\ldots,\rho_{k_n}$ contained in $\gs_j$.
	
To show the same for other $1$-dimensional cones, let $\gs_{j'}\in\gD_X$ be another maximal cone, and define $\tilde{\gs}_{j'},U_{j'},\tilde{U}_{j'},v_{k'_l},u_{k'_l},z'_l,z'$ in the same way.  Similarly to above, $\big\{\pdv{z'_1},\ldots,\pdv{z'_n},z'\pdv{z'}\big\}$ is a local frame of $\calE$ on $U_{j'}$ consisting of $T$-eigensections with
	$$
	\begin{dcases}
		(t_1,\ldots,t_n)\cdot
		\pdv{z'_l}
		=
		\chi^{u_{k'_l}}\!(t_1,\ldots,t_n)
		\pdv{z'_l}
		&
		\text{for }l=1,\ldots,n,
		\\
		(t_1,\ldots,t_n)\cdot
		\left(z'\pdv{z'}\right)
		=
		z'\pdv{z'}
		=
		\chi^{(0,\ldots,0)}(t_1,\ldots,t_n)\cdot
		\left(z'\pdv{z'}\right).
	\end{dcases}
	$$
	
The claim is that $\pdv{z'_l}\big|_{t_0}$ (where $l=1,\ldots,n$) and $z'\pdv{z'}\big|_{t_0}$ are exactly $(v_{k'_l},-1)$ and $(0,\ldots,0,1)$ in terms of the coordinates we have chosen for $E$.  In that case, Lemma~\ref{lem:2.1} will imply that the filtrations $E^{\rho_{k'_l}}(i)$ corresponding to the $1$-dimensional cones $\rho_{k'_l}$ in $\gs_{j'}$ are also as desired.  The proof will then be done since every $1$-dimensional cone is contained in a maximal cone (by the fact that $\gD_X$ is complete).
	
Now, to prove the claim, we first consider the change of basis matrix between the two bases of $E$, $\big\{(v_{k_l},-1),(0,\ldots,0,1)\big\}$ and $\big\{(v_{k'_l},-1),(0,\ldots,0,1)\big\}$.  The change of basis matrix between their dual bases, $\big\{(u_{k_l},0),(u,1)\big\}$ and $\big\{(u_{k'_l},0),(u',1)\big\}$, can be written as
	$$
	\begin{pmatrix}
		(u_{k'_1},0)	\\
		\vdots			\\
		(u_{k'_n},0)	\\[.3ex]
		(u',1)
	\end{pmatrix}
	=
	\begin{pmatrix}
		&			&		&	\rvline	&	0		\\
		&	A		&		&	\rvline	&	\vdots	\\
		&			&		&	\rvline	&	0		\\[.3ex]
		\hline
		b_1	&	\cdots	&	b_n	&	\rvline	&	1
	\end{pmatrix}
	\begin{pmatrix}
		(u_{k_1},0)	\\
		\vdots		\\
		(u_{k_n},0)	\\[.3ex]
		(u,1)
	\end{pmatrix},
	$$
where $A\in\bbZ^{n\times n}$ is the change of basis matrix between $\{u_{k_l}\}$ and $\{u_{k'_l}\}$ and the $b_l$ are integers such that $u'=u+\sum_{l=1}^n b_lu_{k_l}$.  Taking the dual of this relation, we get
	$$
	\begin{pmatrix}
		(v_{k'_1},-1)	\\
		\vdots			\\
		(v_{k'_n},-1)	\\[.3ex]
		(0,\ldots,0,1)
	\end{pmatrix}
	=
	\begin{bmatrix}
		\begin{pmatrix}
			&			&		&	\rvline	&	0		\\
			&	A		&		&	\rvline	&	\vdots	\\
			&			&		&	\rvline	&	0		\\[.3ex]
			\hline
			b_1	&	\cdots	&	b_n	&	\rvline	&	1
		\end{pmatrix}
		^{\!\!T} \;
	\end{bmatrix}
	^{\!-1}\!\!
	\begin{pmatrix}
		(v_{k_1},-1)	\\
		\vdots			\\
		(v_{k_n},-1)	\\[.3ex]
		(0,\ldots,0,1)
	\end{pmatrix}
	=
	\begin{pmatrix}
		&			&		&	\rvline	&	b_1		\\
		&	A^T		&		&	\rvline	&	\vdots	\\
		&			&		&	\rvline	&	b_n		\\[.3ex]
		\hline
		0	&	\cdots	&	0	&	\rvline	&	1
	\end{pmatrix}
	^{\!\!\!-1}\!
	\begin{pmatrix}
		(v_{k_1},-1)	\\
		\vdots			\\
		(v_{k_n},-1)	\\[.3ex]
		(0,\ldots,0,1)
	\end{pmatrix}
	.
	$$
	
On the other hand, let $U_{jj'}:=U_j\cap U'_j$, and consider the transition function
	$$
	g_{j'j}\colon
	\calE\Big|_{U_j}\Big|_{U_{jj'}}
	\lra
	\calE\Big|_{U_j}\Big|_{U_{jj'}}.
	$$
Note that the transition function $\tilde{g}_{j'j}$ of $\calT_Y(-\log X)$ is given by
	$$
	\begin{dcases}
		\tilde{g}_{j'j}
		\left(\pdv{z_{l_0}}\bigg|_{\tilde{p}}\right)
		=
		\sum_{l'=1}^n\left(
		a_{l'l_0}
		\cdot
		z_{l_0,\tilde{p}}^{-1}
		\cdot
		\prod_{l=1}^nz_{l,\tilde{p}}^{a_{l'l}}
		\cdot
		\pdv{z'_{l'}}\bigg|_{\tilde{p}'}
		\right)
		+
		b_{l_0}
		\cdot
		z_{l_0,\tilde{p}}^{-1}
		\cdot 
		\left(z'\pdv{z'}\right)\bigg|_{\tilde{p}'}
		,
		\\
		\tilde{g}_{j'j}
		\left(\left(z\pdv{z}\right)\bigg|_{\tilde{p}}\right)
		=
		\left(z'\pdv{z'}\right)\bigg|_{\tilde{p}'},
	\end{dcases}
	$$
where $\tilde{p}=(z_{1,\tilde{p}},\ldots,z_{n,\tilde{p}},z_{\tilde{p}})\in\tilde{U}_j\cap\tilde{U}_{j'}$ and $\tilde{p}'$ is the point in $\tilde{U}_{j'}$ corresponding to $\tilde{p}$.  Restricting this to $\calE$ and noting that $t_0\in T\subseteq U_j$ corresponds to the point $(z_1,\ldots,z_n,z)=(1,\ldots,1,0)\in\tilde{U}_j$, we have
	$$
	\begin{dcases}
		g_{j'j}
		\left(\pdv{z_l}\bigg|_{t_0}\right)
		=
		\sum_{l'=1}^n
		\left(a_{l'l}\cdot\pdv{z'_{l'}}\bigg|_{t_0}\right)
		+b_l
		\cdot
		\left(z'\pdv{z'}\right)\bigg|_{t_0}
		&
		\text{for }l_0=1,\ldots,n,
		\\
		g_{j'j}
		\left(\left(z\pdv{z}\right)\bigg|_{t_0}\right)
		=
		\left(z'\pdv{z'}\right)\bigg|_{t_0}
	\end{dcases}
	$$
	or,
	in terms of matrices,
	$$
	\begin{pmatrix}
		\pdv{z_1}\big|_{t_0}			\\
		\vdots							\\
		\pdv{z_n}\big|_{t_0}			\\[1ex]
		\left(z\pdv{z}\right)\big|_{t_0}
	\end{pmatrix}
	=
	\begin{pmatrix}
			&			&		&	\rvline	&	b_1		\\
			&	A^T		&		&	\rvline	&	\vdots	\\
			&			&		&	\rvline	&	b_n		\\[1ex]
			\hline
		0	&	\cdots	&	0	&	\rvline	&	1
	\end{pmatrix}
	\begin{pmatrix}
		\pdv{z'_1}\big|_{t_0}			\\
		\vdots							\\
		\pdv{z'_n}\big|_{t_0}			\\[1ex]
		\left(z'\pdv{z'}\right)\big|_{t_0}
	\end{pmatrix}.
	$$
	Equivalently,
	we have
	$$
	\begin{pmatrix}
		\pdv{z'_1}\big|_{t_0}			\\
		\vdots							\\
		\pdv{z'_n}\big|_{t_0}			\\[1ex]
		\left(z'\pdv{z'}\right)\big|_{t_0}
	\end{pmatrix}
	=
	\begin{pmatrix}
		&			&		&	\rvline	&	b_1		\\
		&	A^T		&		&	\rvline	&	\vdots	\\
		&			&		&	\rvline	&	b_n		\\[1ex]
		\hline
		0	&	\cdots	&	0	&	\rvline	&	1
	\end{pmatrix}
	^{-1}
	\begin{pmatrix}
		\pdv{z_1}\big|_{t_0}			\\
		\vdots							\\
		\pdv{z_n}\big|_{t_0}			\\[1ex]
		\left(z\pdv{z}\right)\big|_{t_0}
	\end{pmatrix},
	$$
which matches the change of basis matrix we computed above.  Thus, we have proved the claim, and the proof is done.
\end{proof}

\section{Affine subspace concentration conditions} \label{sec:ASCC}
{ Let $M\,(\cong\bbZ^n)$ be a rank $n$ lattice, and define $M_{\bbR}:=M\otimes_{\bbZ}\bbR\cong \bbR^n$.  Let $P\subseteq M_{\bbR}$ be a smooth and reflexive lattice polytope with barycenter at the origin.  By replacing $M_{\bbR}$ with its subspace spanned by $P$, we may assume that $P$ is full-dimensional, \textit{i.e.}, $\dim P=n$.

Such a polytope corresponds to an $n$-dimensional smooth Fano toric variety $X=X_P$, together with its anticanonical divisor $-K_X$.  The fan $\gD_X$ of $X$ in $N_{\bbR}$ is the inner normal fan of $P$, where $N$ is the dual lattice of $M$ and $N_{\bbR}:=N\otimes_{\bbZ}\bbR\cong \bbR^n$.

Denote the $1$-dimensional cones in $\gD_X$ by $\rho_1,\ldots,\rho_m$.  For each $k=1,\ldots,m$, let $v_k\in N$ be the primitive generator of $\rho_k$ and $P_k$ the facet of $P$ corresponding to $\rho_k$.
We also define $\vol(P_k)$ to be the lattice volume of the facet $P_k$ with respect to the intersection of $M$ with the affine span of $P_k$.  }

{ We will use the following proposition from \cite{HNS19} which characterizes the stability of toric vector bundles.

\begin{prop}[\textit{cf.}~\protect{\cite[Proposition 2.3]{HNS19}}]
	\label{prop:HNS19}
Let $(X,\calO_X(D))$ be a polarized smooth toric variety and $Q\subseteq M_{\bbR}$ the corresponding polytope.  For each $1$-dimensional cone $\rho$ in the fan of $X$, let $Q^{\rho}$ be the facet of $Q$ corresponding to $\rho$ and $\vol(Q^{\rho})$ the lattice volume of the facet $Q^{\rho}$ with respect to the intersection of\, $M$ with the affine span of\,~$Q^{\rho}$.  A toric vector bundle $\calE$ on $X$ corresponding to filtrations $E^{\rho}(i)$ is $($semi-$)$stable with respect to $\calO_X(D)$ if and only if the following inequality holds for every proper\, $\bbC$-linear subspace $F\subsetneq E$ and $F^{\rho}(i) = E^{\rho}(i) \cap F$: 
	\begin{equation}\label{eqn:ASCC}
		\frac{1}{\dim F}\sum_{i,\rho}i\cdot f^{[\rho]}(i) \cdot \vol(Q^{\rho})
		\overset{(\leq)}{<}
		\frac{1}{\dim E}\sum_{i,\rho}i\cdot e^{[\rho]}(i) \cdot \vol(Q^{\rho}),
		\tag{$*$}
	\end{equation}
where $f^{[\rho]}(i):=\dim F^{\rho}(i)-\dim F^{\rho}(i+1)$ and $e^{[\rho]}(i):=\dim E^{\rho}(i)-\dim E^{\rho}(i+1)$.
\end{prop}
}

{ It turns out that to check the stability of toric vector bundle, it suffices to consider equivariant reflexive subsheaves $\calF$ of $\calE$ corresponding to a $\bbC$-linear subspace $F$ of $E$ together with the filtrations $\{F^{\rho}(i):=F\cap E^{\rho}(i)\}$.  Then the equation (\ref{eqn:ASCC}) is equivalent to the statement that $\mu(\calF)\leq\mu(\calE)$.  }

{ In our setting, we consider the anticanonical polarization $D=-K_X$ and the canonical extension $\calE$ of $\calT_X$ by $\calO_X$ with
  the extension class $c_1(\calT_X)\in\Ext^1(\calO_X,\calT_X)$.  }

\begin{prop} \label{prop:ver1}
Let $E:=N_{\bbC}\oplus\bbC\cong\bbC^{n+1}$.  For every proper\, $\bbC$-linear subspace $F\subsetneq E$, we have 
	$$
	\frac{1}{\dim_{\bbC} F}
	\sum_{k\,:\,(v_k,-1)\in F}\vol(P_k)
	\leq
	\frac{1}{n+1}\sum_{k=1}^m\vol(P_k).
	$$
In addition, whenever  equality holds for some $F$,  equality also holds for some subspace $F'\subsetneq E$ complementary to~$F$.
\end{prop}

\begin{proof}
First, since the barycenter of $P$ is at the origin, by \cite[Corollary 1.3]{WZ04} and \cite[Corollary 5.5]{Mab87}, we have that $X$ admits a K\"ahler--Einstein metric.  Then, applying \cite[Theorem 1.1]{Tia92} and the easy direction of the Donaldson--Uhlenbeck--Yau theorem, we obtain that $\calE$ is polystable with respect to $\calO_X(-K_X)$.  In particular, $\calE$ is semistable with respect to $\calO_X(-K_X)$.
	
Define $F^{\rho_k}(i):=F\cap E^{\rho_k}(i)$.  Then, by Proposition~\ref{prop:fltrE}, for all $k=1,\ldots,m$, we have $F^{\rho_k}(i)=F$ for $i\leq 0$, $F^{\rho_k}(i)=0$ for $i\geq 2$, and
	$$
	F^{\rho_k}(1)=
	\begin{cases}
		\Span_{\bbC}\{(v_k,-1)\}	&\text{if }(v_k,-1)\in F,
		\\
		0							&\text{if }(v_k,-1)\notin F.
	\end{cases}
	$$
Thus, using the notation in Proposition~\ref{prop:HNS19}, we have $f^{[\rho_k]}(i)=0$ for all $i\neq 0,1$ and
	$$
	f^{[\rho_k]}(1)=
	\begin{cases}
		1	&\text{if }(v_k,-1)\in F,
		\\
		0	&\text{if }(v_k,-1)\notin F.
	\end{cases}
	$$
The inequality then follows from Proposition~\ref{prop:HNS19}.
	
To see the second part, note that a proper $\bbC$-linear subspace $F\subsetneq E$ such that  equality holds corresponds to a proper subbundle $\calF$ of $\calE$ that has the same slope as $\calE$.  Since $\calE$ is polystable, $\calE$ has another (proper) subbundle $\calF'$ with the same slope such that $\calE\cong\calF\oplus\calF'$.  Then the subspace $F'$ corresponding to $\calF'$ is exactly the desired subspace.
\end{proof}

Next, we replace the complex vector space $E$ with a real one. 

\begin{cor} \label{cor:ver2}
Let $V:=N_{\bbR}\oplus\bbR\cong\bbR^{n+1}$.  For every proper $\bbR$-linear subspace $W\subsetneq V$, we have 
	$$
	\frac{1}{\dim_{\bbR} W}
	\sum_{k\,:\,(v_k,-1)\in W}\vol(P_k)
	\leq
	\frac{1}{n+1}\sum_{k=1}^m\vol(P_k).
	$$
In addition, whenever  equality holds for some $W$,  equality also holds for some subspace $W'\subsetneq V$ complementary to~$W$.
\end{cor}

\begin{proof}
First, view the $\bbC$-vector space $E$ defined in the previous proposition as $V\otimes_{\bbR}\bbC$.  Embed $V$ into $E$ by $w\mapsto w\otimes 1$.  By abuse of notation, we will also use $V$ to denote its image in $E$, and similarly for (real) subspaces of $V$.

An important observation is that the vectors $(v_k,-1)$ are all in $V$.  Thus, given a $\bbC$-linear subspace $E'\subseteq E$, we have $(v_k,-1)\in E'$ if and only if $(v_k,-1)\in (E'\cap V)$.
	
Now, to prove the inequality, consider $F:=W\otimes_{\bbR}\bbC$, which is a $\bbC$-linear subspace of $E$.  We have $\dim_{\bbC} F=\dim_{\bbR} W$.  Also, by the observation above, we have $(v_k,-1)\in F$ if and only if $(v_k,-1)\in W$.  Then, the inequality simply follows from Proposition~\ref{prop:ver1}.
	
Next, to show the second part of the corollary, suppose the equality holds for some real subspace $W\subsetneq V$.  Similarly to above, consider $F:=W\otimes_{\bbR}\bbC$.  Note that
the equality in Proposition~\ref{prop:ver1} holds for $F$.  Then, the second part of Proposition~\ref{prop:ver1} gives a $\bbC$-linear subspace $F'$ of $E$ which is complementary to $F$ and for which the equality in Proposition~\ref{prop:ver1} holds.
	
Consider $W':=F'\cap V$.  Note that $W'\otimes_{\bbR}\bbC\subseteq F'$, which implies $\dim_{\bbR} W'\leq\dim_{\bbC} F'$.  By the observation above, $\{(v_k,-1)\}\cap F'=\{(v_k,-1)\}\cap W'$.  Therefore, we obtain
	$$
	\frac{1}{\dim W'}
	\sum_{k\,:\,(v_k,-1)\in W'}\vol(P_k)
	\geq
	\frac{1}{\dim F'}
	\sum_{k\,:\,(v_k,-1)\in F'}\vol(P_k)
	=
	\frac{1}{n+1}\sum_{k=1}^m\vol(P_k).
	$$
But we have proved the opposite inequality in the first part.  Hence, equality must hold here.  Moreover, $W'$ is complementary to $W$ since we have $W\cap W'\subseteq F\cap F'=0$ and
	\begin{equation*}\pushQED{\qed} 
	\dim_{\bbR} W'=\dim_{\bbC} F'=n-\dim_{\bbC} F=n-\dim_{\bbR} W.\qedhere
\popQED
	\end{equation*}
\renewcommand{\qed}{}    
\end{proof}

Finally, we deduce the affine subspace concentration conditions.

\begin{thm}[Affine subspace concentration conditions]\label{thm:ASCC}
For every proper affine subspace $A\subsetneq N_{\bbR}\cong\bbR^n$, the following inequality holds: 
	$$
	\frac{1}{\dim_{\bbR} A +1}\sum_{k\,:\,v_k\in A}\vol(P_k)
	\leq
	\frac{1}{n+1}\sum_{k=1}^m \vol(P_k).
	$$

In addition, whenever  equality holds for some $A$,  equality also holds for some affine subspace $A'$ complementary to $A$.
\end{thm}

\begin{proof}
Let $x_{n+1}$ be the last coordinate of $V=N_{\bbR}\oplus\bbR\,(\cong\bbR^{n+1})$, and let $W_0$ be the hyperplane $\{x_{n+1}=0\}$ in~$V$.  Embed $N_{\bbR}$ into $V$ via $v\mapsto (v,-1)$.  By abuse of notation, we also use $N_{\bbR}$ to denote the image of this embedding, \textit{i.e.}, the hyperplane $\{x_{n+1}=-1\}$ in $V$ one unit below $W_0$.  Note that this embedding maps each $v_k\in N_{\bbR}$ to $(v_k,-1)\in V$.
	
There is a one-to-one correspondence between the affine subspaces of $N_{\bbR}$ and the linear subspaces of $V$ that are not contained in $W_0$.  Given a linear subspace $W\subseteq V$ that is not contained in $W_0$, we can get an affine subspace in $N_{\bbR}$ by taking the intersection $A:=W\cap N_{\bbR}$.  On the other hand, given an affine subspace $A\subseteq N_{\bbR}$, we can take the linear span of $A$ (and the origin $0\in V$) to recover the linear subspace $W$.
	
Note that if a linear subspace $W\subseteq V$ corresponds to an affine subspace $A\subseteq N_{\bbR}$, then $\dim_{\bbR} W=\dim_{\bbR} A +1$.  Then, it is easy to see that the inequality here follows from Corollary~\ref{cor:ver2}.
	
The second part also almost follows directly from Corollary~\ref{cor:ver2}, but we need to show that the complementary linear subspace $W'\subseteq V$ obtained from the corollary cannot be contained in $W_0$.  Indeed, since $W_0$ does not contain any of the vectors $(v_k,-1)$, every linear subspace of $W_0$ will give us $0$ on the left-hand side of the inequality in Corollary~\ref{cor:ver2}.  However, the right-hand side of the inequality is positive, so equality cannot hold for any linear subspace of $W_0$.  Therefore, $W'$ is not contained in $W_0$, and its intersection with $N_{\bbR}$ gives the desired complementary affine subspace $A'$.
\end{proof}


\newcommand{\etalchar}[1]{$^{#1}$}

\end{document}